\documentclass[12pt]{amsart}
\usepackage{amsmath,amsthm,amsfonts,amssymb,mathrsfs}
\usepackage[all]{xy}
\usepackage{eufrak}
\usepackage{amssymb}
\usepackage{latexsym}
\usepackage{amsmath,amsthm}
\usepackage{aeguill}
\usepackage{amssymb}
\usepackage{mathrsfs}
\usepackage{hyperref}
\usepackage{etoolbox}
\usepackage{enumerate}

\makeatletter
\pretocmd{\chapter}{\addtocontents{toc}{\protect\addvspace{15\p@}}}{}{}
\pretocmd{\section}{\addtocontents{toc}{\protect\addvspace{3\p@}}}{}{}
\makeatother

\makeatletter
\def\@tocline#1#2#3#4#5#6#7{\relax
  \ifnum #1>\c@tocdepth 
  \else
    \par \addpenalty\@secpenalty\addvspace{#2}%
    \begingroup \hyphenpenalty\@M
    \@ifempty{#4}{%
      \@tempdima\csname r@tocindent\number#1\endcsname\relax
    }{%
      \@tempdima#4\relax
    }%
    \parindent\z@ \leftskip#3\relax \advance\leftskip\@tempdima\relax
    \rightskip\@pnumwidth plus4em \parfillskip-\@pnumwidth
    #5\leavevmode\hskip-\@tempdima
      \ifcase #1
       \or\or \hskip .5em \or \hskip 1em \else \hskip 1.5em \fi%
      #6\nobreak\relax
    \dotfill\hbox to\@pnumwidth{\@tocpagenum{#7}}\par
    \nobreak
    \endgroup
  \fi}
\makeatother
\newcommand{\C}{\mathbb{C}}
\newcommand{\N}{\mathbb{N}}
\newcommand{\Z}{\mathbb{Z}}

\newcommand{\Q}{\mathbb{Q}}
\newcommand{\F}{\mathbb{F}}

\newcommand{\cG}{\mathcal{G}}

\newcommand{\Gal}{\operatorname{Gal}}

\newcommand{\geo}{\operatorname{geo}}

\renewcommand{\sc}{\operatorname{sc}}

\newcommand{\End}{\operatorname{End}}

\newcommand{\GL}{\mathrm{GL}}

\newcommand{\Frob}{\mathrm{Fr}}

\newtheorem{thm}{Theorem}[section]

\newtheorem{cor}[thm]{Corollary}

\newtheorem{prop}[thm]{Proposition}
\newtheorem{lemma}[thm]{Lemma}

\newtheorem{remark}[thm]{Remark}

\newtheorem{innercustomconj}{Conjecture}
\newenvironment{customconj}[1]
  {\renewcommand\theinnercustomconj{#1}\innercustomconj}
  {\endinnercustomconj} 
  
\newtheorem{innercustomthm}{Theorem}
\newenvironment{customthm}[1]
  {\renewcommand\theinnercustomthm{#1}\innercustomthm}
  {\endinnercustomthm}

\begin{document}

\title{Invariant dimensions and maximality of geometric monodromy action}

\thanks{Chun Yin Hui is supported by the National Research Fund, Luxembourg, and cofunded under the Marie Curie Actions of the European Commission (FP7-COFUND)}
\maketitle

\begin{center}
Chun Yin Hui
\end{center}

\vspace{.1in}

\begin{center}
Mathematics Research Unit, University of Luxembourg,\\
6 rue Richard Coudenhove-Kalergi, L-1359 Luxembourg\\
Email: \url{pslnfq@gmail.com}
\end{center}

\begin{abstract}
Let $X$ be a smooth separated geometrically connected variety over $\F_q$ and $f:Y\to X$ a  smooth projective morphism.
We compare the invariant dimensions of the $\ell$-adic representation $V_\ell$ 
and the $\F_\ell$-representation $\bar V_\ell$ of the geometric \'etale fundamental group of $X$ 
arising from the sheaves $R^wf_*\Q_\ell$ and $R^wf_*\Z/\ell\Z$ respectively.
These invariant dimension data is used to deduce a maximality result 
of the geometric monodromy action on $V_\ell$ whenever $\bar V_\ell$
is semisimple and $\ell$ is sufficiently large. 
We also provide examples for $\bar V_\ell$ to be semisimple for $\ell\gg0$.
\end{abstract}

\maketitle

\section{Introduction}
Consider a smooth projective $\F_q$-morphism $f:Y\to X$,
where $X$ is a smooth separated geometrically connected $\F_q$-variety.
Fix a geometric point $\bar x_0:\mathrm{Spec}(\bar\F_q)\to X$.
For any prime $\ell\nmid q$ and integer $w$, $\mathscr{F}_\ell:=R^w f_*\Q_\ell$ is a \emph{lisse}, 
\emph{pure of weight} $w$, $\Q_\ell$-sheaf  
on $X$ \cite{De80} inducing 
an $\ell$-adic representation of the \emph{\'etale fundamental group} $\pi_1^{et}(X):=\pi_1^{et}(X,\bar x_0)$
on the stalk $\mathscr{F}_{\ell,\bar x_0}\cong H^w(Y_{\bar x_0},\Q_\ell)=:V_\ell$,
\begin{equation}\label{1}
\Phi_\ell:\pi_1^{et}(X)\to \GL(V_\ell);
\end{equation}
$\bar{\mathscr{F}}_\ell:=R^w f_*\Z/\ell\Z$ is a locally constant sheaf on $X$ inducing
an $\F_\ell$-representation on the stalk $\bar{\mathscr{F}}_{\ell,\bar x_0}\cong H^w(Y_{\bar x_0},\Z/\ell\Z)=:\bar V_\ell$,
\begin{equation}\label{1'}
\phi_\ell:\pi_1^{et}(X)\to \GL(\bar V_\ell).
\end{equation} 
The \emph{geometric \'etale fundamental group} of $X$, $\pi_1^{et}(X_{\bar\F_q}):=\pi_1^{et}(X_{\bar \F_q},\bar x_0)$, is a normal subgroup 
of $\pi_1^{et}(X)$ satisfying the exact sequence
\begin{equation}\label{2}
1\to \pi_1^{et}(X_{\bar\F_q})\to \pi_1^{et}(X)\to \Gal(\bar \F_q/\F_q)\to 1
\end{equation}
so that any $x\in X(\F_q)$ induces a splitting $i_x$ of (\ref{2}).
The \emph{monodromy group} $\Gamma_\ell$ (resp. $\bar\Gamma_\ell$) and the \emph{geometric monodromy group} $\Gamma_\ell^{\geo}$ 
(resp. $\bar{\Gamma}_\ell^{\geo}$) are defined 
to be the images of $\pi_1^{et}(X)$ and $\pi_1^{et}(X_{\bar\F_q})$ respectively in $\GL(V_\ell)$ (resp. $\GL(\bar V_\ell)$);
their Zariski closures in $\GL_{V_\ell}$, denoted respectively by $\mathbf{G}_\ell$ and $\mathbf{G}_\ell^{\geo}$,
are called the \emph{algebraic monodromy group} and the \emph{algebraic geometric monodromy group} of $\Phi_\ell$.

Since $\Gal(\bar \F_q/\F_q)\cong\hat{\Z}$ is abelian, 
the geometric monodromy groups $\Gamma_\ell^{\geo}$ and $\mathbf{G}_\ell^{\geo}$ are of particular interest by (\ref{2}). 
Deligne has proved that the identity component of $\mathbf{G}_\ell^{\geo}$ 
is a semisimple subgroup of $\GL_{V_\ell}$ \cite[Cor. 1.3.9, Thm. 3.4.1(iii)]{De80}.
Determining $\Gamma_\ell^{\geo}$ (or $\bar{\Gamma}_\ell^{\geo}$) and $\mathbf{G}_\ell^{\geo}$ 
for families of curves (elliptic \cite{Ha08}; hyperelliptic \cite{La90s},\cite{Yu96},\cite{AP07}; trielliptic \cite{AP07}) is of independent interest and also has applications to the arithmetic of function fields (see \cite{Yu96},\cite{Ac06}) 
and arithmetic geometry (see \cite{Ch97},\cite{Ko06a,Ko06b,Ko06c,Ko08}) over function fields.
A crucial point is that for all sufficiently large $\ell$, 
the geometric monodromy $\Gamma_\ell^{\geo}$ is a \emph{large} compact subgroup of 
$\mathbf{G}_\ell^{\geo}(\Q_\ell)$.
The motivation of this paper is to investigate the following 
large geometric monodromy conjecture. Let 
$\pi:\mathbf{G}_\ell^{\sc}\to \mathbf{G}_\ell^{\geo}$ be the natural morphism
such that $\mathbf{G}_\ell^{\sc}$
is the universal cover of the identity component of $\mathbf{G}_\ell^{\geo}$. 

\begin{customconj}{1}\label{conj}
Let $\Phi_\ell$ be the $\ell$-adic representation defined in (\ref{1}).
Then $\pi^{-1}(\Gamma_\ell^{\geo})$ is a hyperspecial maximal compact subgroup of 
$\mathbf{G}_\ell^{\sc}(\Q_\ell)$ whenever $\ell$ is sufficiently large.
\end{customconj}

Let us make a detour to the characteristic zero case.
Suppose $f:Y\to X$ is not defined over $\F_q$, but over a subfield $K$ of $\C$. 
Denote the $w$th Betti cohomology $H^w(Y_{\bar x_0}(\C),\Q)$ by $V$, which is acted
on by the topological fundamental group $\pi_1(X(\C))$.
Since the geometric representation $\Phi_\ell:\pi_1^{et}(X_{\bar K})\to \GL(V_\ell)$
is arising from $\Phi:\pi_1(X(\C))\to \GL(V)$ by the comparison theorem between Betti 
and \'etale cohomologies \cite[XII]{SGA1},\cite[XVI]{SGA4} and 
the identity component of algebraic monodromy group of $\Phi$ is semisimple over $\Q$ \cite[Cor. 4.2.9]{De71},
the geometric monodromy $\Gamma_\ell^{\geo}$ is large in $\mathbf{G}_\ell^{\geo}(\Q_\ell)$ 
for $\ell\gg0$, thanks to \cite{MVW84}. 
On the other hand, $\pi_1^{et}(X)$ satisfies (\ref{2}) with $\F_q$ replaced with $K$.
Since $\Gal(\bar K/K)$ is non-abelian,
the monodromy groups $\Gamma_\ell\subset\mathbf{G}_\ell(\Q_\ell)$ are complicated 
and carry a lot of arithmetic information.
If $K$ is a number field and $X=\mathrm{Spec}(K)$, then $\Phi_\ell$ is a Galois representation of $K$ arising from 
the smooth projective variety $Y/K$ and
the largeness of 
$\Gamma_\ell$ in $\mathbf{G}_\ell(\Q_\ell)$ for $\ell\gg0$
follows from the remarkable conjectures of 
Hodge, Grothendieck, Tate, Mumford-Tate, and Serre \cite[$\mathsection11$]{Se94}, see also \cite[$\mathsection5$]{HL15a}.
The prototypical result in this direction is due to Serre \cite{Se72}, which states that for any non-CM elliptic curve $Y$, 
the monodromy $\Gamma_\ell$ on $H^1$ is $\GL_2(\Z_\ell)$ for all sufficiently large $\ell$,
see also \cite{Ri76,Ri85},\cite{Se85},\cite{BGK03,BGK06,BGK10},\cite{Ha11} for certain abelian varieties;
\cite{HL15b} for arbitrary abelian varieties; 
\cite{Se98} for abelian representations; 
\cite{HL14} for type A representations; 
and partial results \cite{La95a},\cite{Hu14} for arbitrary varieties.
To get large Galois monodromy, one always needs handles on the invariants of $V_\ell$ and $\bar V_\ell$.
For example when $Y$ is an abelian variety and $w=1$, Faltings has proved that
the Galois invariants of $V_\ell\otimes V_\ell^*$ and $\bar V_\ell\otimes \bar V_\ell^*$
depend essentially on the endomorphism ring $\End(Y_{\bar K})$ if $\ell$ is sufficiently large \cite{Fa83},\cite{FW84} (the Tate conjecture).
Since the Tate conjecture remains largely open, large Galois monodromy is presumably difficult.

Back to our setting $f:Y\to X$ over $\F_q$, the main idea of this paper is that
there is a \textit{cohomological} way, without resorting to the Tate conjecture, to compare the geometric invariant dimensions of
$V_\ell^{\otimes m}$ and $\bar V_\ell^{\otimes m}$ for sufficiently large $\ell$ and sufficiently many $m$.

\begin{customthm}{2}
For any $m\in\N$, if $\ell$ is sufficiently large, then 
$$\dim_{\F_\ell}(\bar V_\ell^{\otimes m})^{\pi_1^{et}(X_{\bar\F_q})}=\dim_{\Q_\ell}(V_\ell^{\otimes m})^{\pi_1^{et}(X_{\bar\F_q})}.$$
\end{customthm}

This is accomplished in $\mathsection2$ first assuming $X$ is a curve by \'etale cohomology theory \cite{SGA4,SGA4.5,SGA5},\cite{Mi80},\cite{FK87},\cite{Fu11} and the 
remarkable theorems of Deligne \cite{De74b,De80}, Gabber \cite{Ga83}, and de Jong \cite{dJ96}, the general case
then follows from that by space filling curves \cite{Ka99} and $\ell$-independence of $\mathbf{G}_\ell$ \cite{Ch04}.

\begin{customthm}{3}
If $\phi_\ell$ is semisimple for all sufficiently large $\ell$, then Conjecture \ref{conj} holds.
\end{customthm}

Theorem $3$ is proved in $\mathsection3$ by a recent result of Cadoret and Tamagawa on $\bar{\Gamma}_\ell^{\geo}$ \cite{CT15}, the group theoretic techniques we employed and developed in \cite{HL14}, and exploiting the invariant dimension data 
(Theorem $2$ and Corollary \ref{sym}).
The $\F_\ell$-semisimplicity hypothesis of Theorem $3$ holds if $X$ is a curve 
and the fibers of $f$ are curves or abelian varieties \cite{Za74a,Za74b}.
It is suggestive that the hypothesis holds in general because the invariant dimensions of $\Gamma_\ell^{\geo}$ and 
$\bar\Gamma_\ell^{\geo}$ are alike (Theorem $2$) and $\Gamma_\ell^{\geo}$ is semisimple on $V_\ell$.
Nevertheless, we provide in $\mathsection4$ some examples for the hypothesis to hold.

\section{Invariant dimensions}
The notation in $\mathsection1$ remains in force. 
Embed $\bar\Z[\frac{1}{q}]$ into $\bar\Z_\ell$ with unique maximal ideal $\mathfrak{m}_\ell$.
The common dimension of $V_\ell$ for all $\ell$ (not dividing $q$) is also equal 
to the common dimension of $\bar V_\ell$ for all sufficiently large $\ell$ \cite{Ga83}.
Then whenever $\dim_{\F_\ell}\bar V_\ell=\dim_{\Q_\ell}V_\ell$,  one obtains 
\begin{equation}\label{geq}
\dim_{\F_\ell}(\bar V_\ell^{\otimes m})^{\pi_1^{et}(X_{\bar\F_q})}\geq\dim_{\Q_\ell}(V_\ell^{\otimes m})^{\pi_1^{et}(X_{\bar\F_q})}
\end{equation}
for any $m\in\N$ by identifying $\Gamma_\ell^{\geo}$ as a subgroup of $\GL(H^i(Y_{\bar x_0},\Z_\ell))$, the reduction map
$\GL(H^i(Y_{\bar x_0},\Z_\ell))\to\GL(\bar V_\ell)$, and Lemma \ref{summand}.

\begin{lemma}\label{summand}
Let $F$ be a characteristic $0$ non-Archimedean local field with $\mathscr{O}_F$ the ring of integers.
Let $M$ be a free $\mathscr{O}_F$-module of finite rank. If $W$ is an $F$-subspace of $M\otimes_{\mathscr{O}_F} F$, then 
$W\cap M$ is a direct summand of $M$.
\end{lemma}

\begin{proof}
Since $\mathscr{O}_F$ is a PID and $\mathscr{O}_F/I$ is finite for any non-zero ideal $I$, 
the finitely generated module $M/W\cap M$ is a direct sum of a free submodule and 
a torsion submodule. If $x\in M$ maps to a torsion element in $M/W\cap M$, then $k\cdot x\in W\cap M$ for some $k\in\N$.
This implies $x\in W$ because $F$ is of characteristic $0$. Hence, $M/W\cap M$ is free and $W\cap M$ is a direct summand of $M$.
\end{proof}
Let $d$ be the dimension of $X$.
Since $X$ is smooth, $\mathscr{F}_\ell$ is lisse, and $\bar{\mathscr{F}}_\ell$ is locally constant, 
we obtain perfect pairings by Poincar\'e duality \cite[VI Thm. 11.1]{Mi80} which is compatible with the action of \emph{geometric Frobenius} $\Frob_q$:
\begin{align}\label{Poincare}
\begin{split}
H^i(X_{\bar \F_q},\mathscr{F}_\ell)&\times H^{2d-i}_c(X_{\bar \F_q},\mathscr{F}_\ell^\vee)\to \Q_\ell(-d);\\
H^i(X_{\bar \F_q},\bar{\mathscr{F}}_\ell)&\times H^{2d-i}_c(X_{\bar \F_q},\bar{\mathscr{F}}_\ell^\vee)\to \F_\ell(-d).
\end{split}
\end{align}
The geometric invariants admit the following descriptions:
\begin{align}\label{des}
\begin{split}
V_\ell^{\pi_1^{et}(X_{\bar\F_q})}&=H^0(X_{\bar \F_q},\mathscr{F}_\ell);\\
\bar V_\ell^{\pi_1^{et}(X_{\bar\F_q})}&=H^0(X_{\bar \F_q},\bar{\mathscr{F}}_\ell).
\end{split}
\end{align}
Without loss of generality, assume $x_0$ is an $\F_q$-rational point of $X$ that
induces a \emph{splitting} of (\ref{2}). 
Then the \emph{multiset} $A'$ of the 
$\Frob_q$-eigenvalues on $V_\ell:=H^w(Y_{\bar x_0},\Q_\ell)$ 
are independent of $\ell$ \cite{De74b} and pure of weight $w$ \cite{De80}.
It follows that the eigenvalues on
$H^0(X_{\bar \F_q},\mathscr{F}_\ell)$ belong to 
those on $V_\ell$ by the splitting and (\ref{des}).
One also sees by the same token that the eigenvalues on $H^0(X_{\bar \F_q},\bar{\mathscr{F}}_\ell)$ belong to 
the reduction modulo $\mathfrak{m}_\ell$ of the eigenvalues on $V_\ell$ whenever $\dim_{\F_\ell}\bar V_\ell=\dim_{\Q_\ell}V_\ell$.
Define $A$ to be the following multiset:
\begin{equation}\label{setA}
A:=\{q^d\alpha^{-1}:~\alpha\in A'\}.
\end{equation}
We conclude by (\ref{Poincare}) and above that the numbers in $A$ are pure of weight $2d-w$, 
the eigenvalues of $H^{2d}_c(X_{\bar \F_q},\mathscr{F}_\ell^\vee)$
is a sub-multiset of $A$, and the eigenvalues of $H^{2d}_c(X_{\bar \F_q},\bar{\mathscr{F}}_\ell^\vee)$
is a sub-multiset of the reduction modulo $\mathfrak{m}_\ell$  of $A$ for $\ell\gg0$.

\begin{customthm}{2}\label{inv}
For any $m\in\N$, if $\ell$ is sufficiently large, then 
$$\dim_{\F_\ell}(\bar V_\ell^{\otimes m})^{\pi_1^{et}(X_{\bar\F_q})}=\dim_{\Q_\ell}(V_\ell^{\otimes m})^{\pi_1^{et}(X_{\bar\F_q})}.$$
\end{customthm}

\begin{proof}
\textit{\textbf{Step I}}. Assume $X$ is a  (geometrically connected) curve, i.e., $d=1$. If $U$ is an affine open subscheme of $X$ containing $x_0$, 
then $\pi_1^{et}(U)$ surjects onto $\pi_1^{et}(X)$ \cite[Prop. 3.3.4(i)]{Fu11} and we obtain a commutative diagram:
\begin{equation}\label{square}
\begin{aligned}
\xymatrix{
\pi_1^{et}(U) \ar@{->>}[d] \ar[r] & \GL(V_\ell) \ar[d]^{=} \\
\pi_1^{et}(X) \ar[r]  & \GL(V_\ell)}
\end{aligned}
\end{equation}
Hence, we may further assume $X$ is an affine curve. 

\textit{\textbf{Step II}}. Let $e>0$ be the relative dimension of $f:Y\to X$. Then the dimension of $Y$ is $e+1$.
Let $Y^c$ be a compactification of $Y_{\bar\F_q}$. 
Then $Y^c$ admits a \emph{simplicial scheme} $Y.$ projective and smooth over $\bar\F_q$ and 
an augmentation $Y.\to Y^c$ which is a \emph{proper hypercovering} of $Y^c$ (see \cite[$\mathsection1$]{dJ96}).
This induces a spectral sequence
\begin{equation}\label{ss1}
E^{i,j}_1:=H^j(Y_i,\Z/\ell\Z)\Rightarrow H^{i+j}(Y^c,\Z/\ell\Z)
\end{equation}
by \cite[(6.3), Thm. 7.9]{Co03} (see also \cite{De74a}). 
Let $B'$ be the multiset consisting of all the $\Frob_q$-eigenvalues on $H^j(Y_i,\Q_\ell)$ 
for all $(i,j)\in\Z_{\geq 0}^2$ satisfying $i+j=2(e+1)-(1+w)$. 
Since $Y_i$ is smooth projective for all $i$, 
the multiset $B'$ is mixed of weight $\leq 2(e+1)-(1+w)$ and is independent of $\ell$ \cite{De74b,De80}.
Since there are only finitely many such $(i,j)$, 
 the $\Frob_q$-eigenvalues on 
$$H^{2(e+1)-(1+w)}(Y^c,\Z/\ell\Z)=:H^{2(e+1)-(1+w)}_c(Y_{\bar\F_q},\Z/\ell\Z)$$
belong to the reduction (modulo $\mathfrak{m}_\ell$) of $B'$ for $\ell\gg0$ 
by \cite{Ga83} and the biregular spectral sequence (\ref{ss1}). 
Since $Y$ is smooth, the reduction (modulo $\mathfrak{m}_\ell$) of the multiset (mixed of weight $\geq 1+w$)
\begin{equation*}
B'':=\{q^{e+1}\beta^{-1}:~\beta\in B'\}
\end{equation*}
contains all the $\Frob_q$-eigenvalues on $H^{1+w}(Y_{\bar\F_q},\Z/\ell\Z)$ for $\ell\gg0$ by Poincare duality.
Since the spectral sequence 
\begin{equation*}
E^{i,j}_2:=H^i(X_{\bar\F_q},R^jf_*\Z/\ell\Z)\Rightarrow H^{i+j}(Y_{\bar\F_q},\Z/\ell\Z)
\end{equation*}
degenerates on page $2$ (as $X$ is an affine curve), $E^{1,w}_2=H^1(X_{\bar\F_q},\bar{\mathscr{F}}_\ell)$
is a sub-quotient of $H^{1+w}(Y_{\bar\F_q},\Z/\ell\Z)$. Thus, the eigenvalues on 
$H^1(X_{\bar\F_q},\bar{\mathscr{F}}_\ell)$ belong to the reduction of $B''$ for $\ell\gg0$.
Then we conclude that the multiset (mixed of weight $\leq 1-w$) 
\begin{equation}\label{setB}
B:=\{q\beta^{-1}:~\beta\in B''\}
\end{equation}
after reduction contains all the eigenvalues 
on $H^1_c(X_{\bar\F_q},\bar{\mathscr{F}}_\ell^\vee)$  for $\ell\gg0$ by $X$ smooth and Poincare duality again.

\textit{\textbf{Step III}}. By the Lefschetz trace formula on the lisse sheaf $\mathscr{F}_\ell^\vee$ and 
the locally constant sheaf $\bar{\mathscr{F}}_\ell^\vee$ on $X_{\bar\F_q}$ \cite[VI Thm. 13.4]{Mi80}, we obtain
\begin{align}\label{trace}
\begin{split}
\sum_{x\in X(\F_{q^k})}\mathrm{Tr}(\Frob_{q}^k:H^w(Y_{\bar x},\Q_\ell)^\vee)
&=\sum_{i=0}^2(-1)^i\mathrm{Tr}(\Frob_{q}^k:H^i_c(X_{\bar\F_q},\mathscr{F}_\ell^\vee));\\
\sum_{x\in X(\F_{q^k})}\mathrm{Tr}(\Frob_{q}^k:H^w(Y_{\bar x},\Z/\ell\Z)^\vee)
&=\sum_{i=0}^2(-1)^i\mathrm{Tr}(\Frob_{q}^k:H^i_c(X_{\bar\F_q},\bar{\mathscr{F}}_\ell^\vee))
\end{split}
\end{align}
for all $k\in\N$. Since $Y_{\bar x}$ is smooth projective, the Frobenius action on $H^w(Y_{\bar x},\Z/\ell\Z)^\vee$
factors through $H^w(Y_{\bar x},\Q_\ell)^\vee$ for $\ell\gg0$ \cite{Ga83}. 
Hence, the reduction of the first local sum
is equal to the second local sum for $\ell\gg0$. Since $H^0_c=0$ by $X$ affine, we obtain
\begin{equation}\label{trace2}
\sum_{i=1}^2(-1)^i\overline{\mathrm{Tr}(\Frob_{q}^k:H^i_c(X_{\bar\F_q},\mathscr{F}_\ell^\vee))}=\sum_{i=1}^2(-1)^i\mathrm{Tr}(\Frob_{q}^k:H^i_c(X_{\bar\F_q},\bar{\mathscr{F}}_\ell^\vee))
\end{equation}
for $\ell\gg0$ by reduction.
Denote the reductions of $A$ (\ref{setA}) and $B$ (\ref{setB}) by $\bar A$ and $\bar B$ respectively and the following:
\begin{itemize}
\item $\{\alpha_1,...,\alpha_r\}\subset\bar A$ the multiset of the $\Frob_q$-eigenvalues on $H^2_c(X_{\bar\F_q},\bar{\mathscr{F}}_\ell^\vee)$; 
\item $\{\beta_1,...,\beta_s\}\subset\bar B$ the multiset of the $\Frob_q$-eigenvalues on $H^1_c(X_{\bar\F_q},\bar{\mathscr{F}}_\ell^\vee)$;
\item $\{a_1,...,a_t\}\subset\bar A$ the multiset of reduction of the $\Frob_q$-eigenvalues on $H^2_c(X_{\bar\F_q},\mathscr{F}_\ell^\vee)$;
\item $\{b_1,...,b_u\}$ the multiset of reduction of the $\Frob_q$-eigenvalues on $H^1_c(X_{\bar\F_q},\mathscr{F}_\ell^\vee)$.
\end{itemize}
Note that $r,t\leq |A|$, $s\leq |B|$, and the number $u$ is independent of $\ell$ \cite{Ka83}.
It follows from (\ref{trace2}) that the above eigenvalues (in $\bar\F_\ell^*$) satisfy
\begin{equation}\label{trace3}
a_1^k+\cdots+a_t^k+\beta_1^k+\cdots+\beta_s^k=\alpha_1^k+\cdots+\alpha_r^k+b_1^k+\cdots+b_u^k
\end{equation}
for all $k\in\N$. If $\ell>|A|+\mathrm{max}\{|B|,u\}$, then by Lemma \ref{trick} the two multisets coincide:
\begin{equation}\label{equal}
\{a_1,...,a_t,\beta_1,...,\beta_s\}=\{\alpha_1,...,\alpha_r,b_1,...,b_u\}.
\end{equation}
Since $A$ is pure of weight $2-w$ and $B$ is mixed of weight $\leq 1-w$ (Step II), $\bar A\cap\bar B= \emptyset$
for $\ell\gg0$ which implies $t\geq r$ by (\ref{equal}).
Together with (\ref{geq}),(\ref{Poincare}), and (\ref{des}), we obtain 
\begin{equation}\label{m=1}
\dim_{\F_\ell}\bar V_\ell^{\pi_1^{et}(X_{\bar\F_q})}=\dim_{\Q_\ell}V_\ell^{\pi_1^{et}(X_{\bar\F_q})}
\end{equation}
for all sufficiently large $\ell$.

\textit{\textbf{Step IV}}. Since $f:Y\to X$ is smooth projective, the natural morphism
$$Y^{[m]}:=\overbrace{Y\times_X Y\times_X \cdots\times_X Y}^{m~\mathrm{terms}} \to X$$
is still smooth projective with the fiber 
\begin{equation}\label{fiber}
(Y^{[m]})_{\bar x_0}=\prod^m Y_{\bar x_0}
\end{equation}
inducing the representations $W_\ell:=H^{mw}((Y^{[m]})_{\bar x_0},\Q_\ell)$ and $\bar W_\ell:=H^{mw}((Y^{[m]})_{\bar x_0},\Z/\ell\Z)$
of $\pi^{et}_1(X)$. For all sufficiently large $\ell$, we have
\begin{equation}\label{invw}
\dim_{\F_\ell}\bar W_\ell^{\pi_1^{et}(X_{\bar\F_q})}=\dim_{\Q_\ell}W_\ell^{\pi_1^{et}(X_{\bar\F_q})}
\end{equation}
by (\ref{m=1}). Since the representation $V_\ell^{\otimes m}$ (resp. $\bar V_\ell^{\otimes m}$) 
is a direct summand of the representation $W_\ell$ (resp. $\bar W_\ell$)
by (\ref{fiber}) and the K$\mathrm{\ddot{u}}$nneth isomorphism, we obtain by (\ref{invw}) that
\begin{equation}\label{mgen}
\dim_{\F_\ell}(\bar V_\ell^{\otimes m})^{\pi_1^{et}(X_{\bar\F_q})}=\dim_{\Q_\ell}(V_\ell^{\otimes m})^{\pi_1^{et}(X_{\bar\F_q})}
\end{equation}
holds for all sufficiently large $\ell$. This proves Theorem \ref{inv} when $X$ is a curve.

\textit{\textbf{Step V}}. For general smooth geometrically connected $X$, 
it suffices to prove Theorem \ref{inv}
for quasi-projective $X$ (see (\ref{square})).
If $C\subset X$ (containing $x_0$) is a smooth geometrically connected curve over $\F_q$, then 
$$\Psi_{\ell}:\pi_1^{et}(C)\to\GL(V_{\ell})$$
factors through $\Phi_\ell$ for all $\ell$.
Denote by $\Lambda_\ell^{\geo}$ and $\mathbf{H}_{\ell}^{\geo}$ respectively the geometric monodromy group 
and the algebraic geometric monodromy group of $\Psi_\ell$.
Choose $\ell_0$ such that the dimension of $\mathbf{G}_{\ell_0}^{\geo}$ is the largest. 
By \cite[Cor. 7, Thm. 8]{Ka99}, there exists a space filling curve $C\subset X$ 
(smooth, geometrically connected, containing $x_0$, over $\F_q$) 
satisfying 
\begin{equation}\label{monoeq}
\mathbf{H}_{\ell_0}^{\geo}= \mathbf{G}_{\ell_0}^{\geo}.
\end{equation}
Since the system $\{\Psi_\ell\}$ is pure of weight $w$ and is semisimple on 
the geometric \'etale fundamental group $\pi_1^{et}(C_{\bar\F_q})$ (Deligne), 
the identity component 
$(\mathbf{H}_{\ell}^{\geo})^\circ$ (semisimple) is isomorphic to the derived group of the identity component of 
the algebraic monodromy group of the semisimplification of $\Psi_\ell$ for all $\ell$ by (\ref{2}).
This implies $(\mathbf{H}_{\ell}^{\geo})^\circ\times\C$ 
is independent of $\ell$ by applying \cite[Thm. 1.4]{Ch04}
to the semisimplification of the system $\{\Psi_\ell\}$. 
In particular, the dimension of $\mathbf{H}_{\ell}^{\geo}$
is independent of $\ell$. 
Since we have
$$\dim \mathbf{H}_{\ell}^{\geo}=\dim \mathbf{H}_{\ell_0}^{\geo}=\dim \mathbf{G}_{\ell_0}^{\geo}\geq \dim\mathbf{G}_{\ell}^{\geo}$$
and the groups $\mathbf{H}_{\ell}^{\geo}\subset \mathbf{G}_{\ell}^{\geo}$ (by $\Psi_\ell$ factors through $\Phi_\ell$)
have the same number of connected components (by (\ref{monoeq}) and \cite[Prop. 2.2(iii)]{LP95}) 
for all $\ell$, we obtain $\mathbf{H}_{\ell}^{\geo}=\mathbf{G}_{\ell}^{\geo}$ for all $\ell$.
Since (a) $\Lambda_\ell^{\geo}$ (resp. $\Gamma_\ell^{\geo}$) is Zariski dense in 
$\mathbf{H}_{\ell}^{\geo}$ (resp. $\mathbf{G}_{\ell}^{\geo}$) and (b) $\Psi_\ell$ factors through $\Phi_\ell$,
we obtain 
\begin{align*}
\begin{split}
\dim_{\F_\ell}(\bar V_\ell^{\otimes m})^{\pi_1^{et}(X_{\bar\F_q})}&\stackrel{(b)}{\leq}\dim_{\F_\ell}(\bar V_\ell^{\otimes m})^{\pi_1^{et}(C_{\bar\F_q})}
\stackrel{(\ref{mgen})}{=}\dim_{\Q_\ell}(V_\ell^{\otimes m})^{\pi_1^{et}(C_{\bar\F_q})}\\
\stackrel{(a)}{=}\dim_{\Q_\ell}(V_\ell^{\otimes m})^{\mathbf{H}_{\ell}^{\geo}(\Q_\ell)}&=
\dim_{\Q_\ell}(V_\ell^{\otimes m})^{\mathbf{G}_{\ell}^{\geo}(\Q_\ell)}\stackrel{(a)}{=}
\dim_{\Q_\ell}(V_\ell^{\otimes m})^{\pi_1^{et}(X_{\bar\F_q})}
\end{split}
\end{align*}
for $\ell\gg0$. We are done by (\ref{geq}).
\end{proof}

\begin{lemma}\label{trick}
Suppose $a_1,...,a_m,b_1,...,b_n\in\bar\F_\ell^*$ satisfying $\mathrm{max}\{m,n\}<\ell$ and
\begin{equation}\label{sums}
a_1^k+\cdots +a_m^k=b_1^k+\cdots +b_n^k
\end{equation}
for all $1\leq k\leq \mathrm{max}\{m,n\}$. Then the two multisets $\{a_1,...,a_m\}$ and $\{b_1,...,b_n\}$ coincide.
\end{lemma}

\begin{proof}
First assume $m=n$. Let $x_1,...,x_m$ be indeterminate variables. Denote 
the elementary symmetric polynomials in $x_1,...,x_m$ by $e_1,...,e_m$ and $x_1^k+\cdots +x_m^k$ by $p_k$.
The \emph{Newton's identities} imply 
$$e_1,...,e_m\in\Z[\frac{1}{m!}](p_1,...,p_m).$$ 
Hence, $e_k(a_1,...,a_m)=e_k(b_1,...,b_m)$ for all $1\leq k\leq m$ by (\ref{sums}) and $m<\ell$.
We conclude that $\{a_1,...,a_m\}=\{b_1,...,b_m\}$ by constructing a degree $m$ polynomial in $\bar\F_\ell[t]$
whose roots are exactly $a_1,...,a_m$ (resp. $b_1,...,b_m$).

Suppose $m>n$. Let $b_{n+1},...,b_m$ be all zeros. Then some $a_i$ is zero by the above case, which contradicts $a_i\in\bar\F_\ell^*$. 
\end{proof}

\begin{cor}\label{sym}
For any $m\in\N$, if $\ell$ is sufficiently large, then 
\begin{align*}
\begin{split}
\dim_{\F_\ell}(\mathrm{Sym}^m\bar V_\ell)^{\pi_1^{et}(X_{\bar\F_q})}&=\dim_{\Q_\ell}(\mathrm{Sym}^m V_\ell)^{\pi_1^{et}(X_{\bar\F_q})};\\
\dim_{\F_\ell}(\mathrm{Alt}^m\bar V_\ell)^{\pi_1^{et}(X_{\bar\F_q})}&=\dim_{\Q_\ell}(\mathrm{Alt}^m V_\ell)^{\pi_1^{et}(X_{\bar\F_q})}.
\end{split}
\end{align*}
\end{cor}

\begin{proof}
Since the left hand side is always greater than or equal to the right hand side of the equation and the representations
 $\bar V_\ell^{\otimes m}$ and $V_\ell^{\otimes m}$ contain respectively
$\mathrm{Sym}^m \bar V_\ell$ and $\mathrm{Sym}^m V_\ell$ (resp. $\mathrm{Alt}^m \bar V_\ell$ and $\mathrm{Alt}^m V_\ell$) 
as direct summands, the corollary follows from Theorem \ref{inv}. 
\end{proof}

\section{Maximality}
If $X'_{\bar\F_q}$ is a connected finite \'etale cover of $X_{\bar\F_q}$, 
then $\pi_1^{et}(X'_{\bar\F_q})$ is a finite index subgroup of $\pi_1^{et}(X_{\bar\F_q})$.
Since $X'_{\bar\F_q}\to X_{\bar\F_q}$ is always defined over some finite extension $\F_{q'}$ of $\F_q$ (e.g., $X'_{\F_{q'}}\to X_{\F_{q'}}$) 
which does not affect the geometric monodromy and the restriction of a semisimple representation to a normal subgroup is still semisimple, 
it suffices to prove Theorem \ref{main}
by considering the base change 
$$Y\times_X X'_{\F_{q'}}\to X'_{\F_{q'}}$$
of $f:Y\to X$ by a connected finite Galois \'etale cover $X'_{\F_{q'}}\to X_{\F_{q'}}\to X$. Hence, we assume from now on that 
the algebraic geometric monodromy group $\mathbf{G}_\ell^{\geo}$ is \emph{connected} for all $\ell$
by taking a connected finite \'etale cover of $X$ \cite[Prop. 2.2(ii)]{LP95}.
Let $n$ be the common dimension of $V_\ell$ for all $\ell$, which is also the common dimension of $\bar V_\ell$ for $\ell\gg0$.

\begin{customthm}{3}\label{main}
If $\phi_\ell$ is semisimple for all sufficiently large $\ell$, then Conjecture \ref{conj} holds.
\end{customthm}

\begin{proof}
\textit{\textbf{Step I}}. For any subgroup $\bar\Gamma$ of $\GL_n(\F_\ell)$, denote by $\bar\Gamma^+$ 
the (normal) subgroup of $\bar\Gamma$ that is generated by $\bar\Gamma[\ell]$, the subset of 
order $\ell$ elements of $\bar\Gamma$.
By taking some connected finite Galois \'etale cover of $X$,
we may assume $\bar\Gamma_\ell^{\geo}=(\bar\Gamma_\ell^{\geo})^+$ \cite[Prop. 3.2, Thm. 1.1]{CT15},
$\bar\Gamma_\ell^{\geo}$ is semisimple on $\bar V_\ell$, and $\mathbf{G}_\ell^{\geo}$ is connected
for all sufficiently large $\ell$.
Since $n=\dim_{\F_\ell}\bar V_\ell$ for $\ell\gg0$,  
 there exists an \emph{exponentially generated} subgroup $\bar{\mathbf{S}}_\ell$ 
 of $\GL_{\bar V_\ell}$ such that 
$\bar\Gamma_\ell^{\geo}=\bar{\mathbf{S}}_\ell(\F_\ell)^+$ for all $\ell\gg 0$ by Nori \cite[Thm. B]{No87}.
The Nori subgroup $\bar{\mathbf{S}}_\ell$ is connected and an extension of semisimple by unipotent.
Since $\bar\Gamma_\ell^{\geo}$ is semisimple on $\bar V_\ell$ for $\ell\gg0$, 
$\bar{\mathbf{S}}_\ell$ is connected semisimple for $\ell\gg0$. Let $\bar{\mathbf{S}}_\ell^{\sc}\to \bar{\mathbf{S}}_\ell$ 
be the universal covering of $\bar{\mathbf{S}}_\ell$. The representation 
$$\bar{\mathbf{S}}_\ell^{\sc}\times\bar\F_\ell\to \bar{\mathbf{S}}_\ell\times\bar\F_\ell\hookrightarrow\GL_{\bar V_\ell\times\bar\F_\ell}$$
can be lifted to a representation of some simply-connected \emph{Chevalley scheme} over $\Z$ 
for all $\ell\gg0$ \cite{Se86} 
(see \cite[Thm. 27]{EHK12}),
\begin{equation}\label{Chevalley}
\rho_{\ell,\Z}:\mathbf{H}_{\ell,\Z}\to \GL_{V_\Z}.
\end{equation}

\textit{\textbf{Step II}}.  
We would like to study the invariants of $\bar{\mathbf{S}}_\ell$ on $\bar V_\ell^{\otimes m}\otimes\bar\F_\ell$.
Let us recall the construction of $\bar{\mathbf{S}}_\ell$. Define $\mathrm{exp}(x)$ and $\mathrm{log}(x)$ by
\begin{equation}\label{explog1}
\mathrm{exp}(x)=\sum_{i=0}^{\ell-1}\frac{x^i}{i!}\hspace{.1in}\mathrm{and}\hspace{.1in}
\mathrm{log}(x)=-\sum_{i=1}^{\ell-1}\frac{(1-x)^i}{i}.
\end{equation}
For all sufficiently large $\ell$, $\bar{\mathbf{S}}_\ell$ is the Zariski closure in $\GL_{\bar V_\ell}\cong\GL_{n,\F_\ell}$ of the subgroup generated by the 
one-parameter subgroup
\begin{equation}\label{explog2}
t\mapsto x^t:=\mathrm{exp}(t\cdot\mathrm{log}(x))
\end{equation}
for all $x\in\bar\Gamma_\ell^{\geo}[\ell]$ (the order $\ell$ elements) \cite{No87}.
When $\ell>n$, $x$ is unipotent and $\mathrm{log}(x)$ is nilpotent by (\ref{explog1}). 
Identify $\bar V_\ell\otimes\bar\F_\ell$ with $\bar\F_\ell^n$, then every entry of the matrix $x^t\in \GL_n(\bar\F_\ell[t])$
is a polynomial of degree less than $n^2$ by (\ref{explog1}) and (\ref{explog2}).
Similarly, the action of $x^t$ on $\bar V_\ell^{\otimes m}\otimes\bar\F_\ell$ can be identified with an element of $\GL_{n^m}(\bar\F_\ell[t])$
whose entries are polynomials of degree less than $n^2m$.
Consider an invariant $v\in (\bar V_\ell^{\otimes m})^{\pi_1^{et}(X_{\bar\F_q})}=(\bar V_\ell^{\otimes m})^{\bar\Gamma_\ell^{\geo}}$, then the equation in $\bar\F_\ell[t]^{n^m}$ below
$$x^t\cdot v=v$$
has at least $\ell$ distinct roots $t=0,1,...,\ell-1$ because $id,x,...,x^{\ell-1}\in \bar\Gamma_\ell^{\geo}$. This implies $x^t\cdot v\equiv v$ when $\ell\geq n^2m$.
Hence, we obtain $v\in (\bar V_\ell^{\otimes m}\otimes\bar\F_\ell)^{\bar{\mathbf{S}}_\ell}$ when $\ell\geq n^2m$ 
by the construction of $\bar{\mathbf{S}}_\ell$ .
Since $\bar\Gamma_\ell^{\geo}=(\bar\Gamma_\ell^{\geo})^+$ \cite{CT15}
is a subgroup of $\bar{\mathbf{S}}_\ell$ for $\ell\gg0$, we obtain
\begin{equation*}
\dim_{\F_\ell} (\bar V_\ell^{\otimes m})^{\pi_1^{et}(X_{\bar\F_q})}    =\dim_{\bar\F_\ell}(\bar V_\ell^{\otimes m}\otimes\bar\F_\ell)^{\bar{\mathbf{S}}_\ell}
\end{equation*}
for $\ell\gg0$. It follows that 
\begin{equation*}
\dim_{\F_\ell} ((\oplus^n\bar V_\ell)^{\otimes m})^{\pi_1^{et}(X_{\bar\F_q})}    =\dim_{\bar\F_\ell}((\oplus^n\bar V_\ell)^{\otimes m}\otimes\bar\F_\ell)^{\bar{\mathbf{S}}_\ell}
\end{equation*}
for $\ell\gg0$. By Corollary \ref{sym} and embedding $\Q_\ell$ into $\C$, we conclude for $\ell\gg0$ that
\begin{equation}\label{Noriinv}
\dim_{\C} (\mathrm{Sym}^m (\oplus^n V_\ell)\otimes \C)^{\mathbf{G}_\ell^{\geo}}=\dim_{\bar\F_\ell}(\mathrm{Sym}^m(\oplus^n\bar V_\ell)\otimes\bar\F_\ell)^{\bar{\mathbf{S}}_\ell}.
\end{equation}

\textit{\textbf{Step III}}. Denote the base change of (\ref{Chevalley}) to $\C$ by $\rho_{\ell,\C}:\mathbf{H}_{\ell,\C}\to \GL_{V_\C}$.
For fixed $m\in\N$, we obtain by Step I and (\ref{Noriinv}) that for $\ell\gg0$,
\begin{equation}\label{Ceqt}
\dim_{\C} (\mathrm{Sym}^m (\oplus^n V_\ell)\otimes \C)^{\mathbf{G}_\ell^{\geo}} =\dim_{\C}(\mathrm{Sym}^m (\oplus^n V_\C))^{\mathbf{H}_{\ell,\C}}.
\end{equation}
Since there are finitely many connected semisimple subgroup of $\GL_{n,\C}$ (up to isomorphism), (\ref{Ceqt}) holds for all $m\in\N$ 
when $\ell$ is sufficiently large. Identify $V_\ell\otimes \C$ with $V_\C$.
Then  the (Noetherian) graded rings
$$R=\C[\oplus^n V_\C]^{\mathbf{G}_\ell^{\geo}}\hspace{.1in}\mathrm{and}\hspace{.1in} R'=\C[\oplus^n V_\C]^{\mathbf{H}_{\ell,\C}}$$
have the same \emph{Hilbert polynomial}, hence the same \emph{Krull dimension} for $\ell\gg0$. 
Since $\dim_{Krull} R=n^2-\dim \mathbf{G}_\ell^{\geo}$ and $\dim_{Krull} R'=n^2-\dim \rho_{\ell,\C}(\mathbf{H}_{\ell,\C})$ 
(for example \cite[$\mathsection0$]{LP90}),  we conclude by the lifting (\ref{Chevalley}) that for all $\ell\gg0$,
\begin{equation}\label{dim}
\dim \mathbf{G}_\ell^{\geo} = \dim\bar{\mathbf{S}}_\ell.
\end{equation}

\textit{\textbf{Step IV}}.
Suppose $\ell\geq 5$.
For any compact subgroup $\Gamma\subset\GL_n(\Q_\ell)$ (resp. $\bar\Gamma\subset\GL_n(\F_\ell)$), 
we defined the \emph{$\ell$-dimension} $\dim_\ell\Gamma$ (resp. $\dim_\ell\bar\Gamma$) in \cite[$\mathsection2$]{HL14}
satisfying the following properties:
\begin{enumerate}[(i)]
\item $\dim_\ell$ is additive on short exact sequences;
\item $\dim_\ell$ vanishes for pro-solvable groups and 
finite simple groups that are not of Lie type in characteristic $\ell$;
\item if $\bar\Gamma$ is a finite simple group of Lie type in characteristic $\ell$, then there exists 
some connected adjoint semisimple group $\bar{\mathbf{S}}/\F_\ell$ such that $\bar\Gamma$ 
is isomorphic to the derived group of  $\bar{\mathbf{S}}(\F_\ell)$
and we define $\dim_\ell\bar\Gamma:=\dim \bar{\mathbf{S}}$.
\end{enumerate}
We obtain for $\ell\gg0$ that 
\begin{equation}\label{ldim}
\dim_\ell\bar\Gamma_\ell^{\geo}=\dim_\ell\bar{\mathbf{S}}_\ell(\F_\ell)^+=\dim_\ell\bar{\mathbf{S}}_\ell(\F_\ell)=\dim\bar{\mathbf{S}}_\ell=\dim \mathbf{G}_\ell^{\geo}
\end{equation}
by Step I, \cite[3.6(v)]{No87}, \cite[Prop. 3(iii)]{HL14}, and (\ref{dim}) respectively for each equality.
Recall the universal covering $\pi:\mathbf{G}_\ell^{\sc}\to \mathbf{G}_\ell^{\geo}$.
Since (a) the kernel and cokernel of $\pi^{-1}(\Gamma_\ell^{\geo})\to \Gamma_\ell^{\geo}$ are abelian and 
(b) the kernel of $\Gamma_\ell^{\geo}\twoheadrightarrow\bar\Gamma_\ell^{\geo}$ 
is pro-solvable (via the reduction map $\GL(H^i(Y_{\bar x_0},\Z_\ell))\to\GL(\bar V_\ell)$ for $\ell\gg0$), 
we obtain by the properties of $\dim_\ell$ that for $\ell\gg0$,
\begin{equation}\label{hopedim}
\dim_\ell\pi^{-1}(\Gamma_\ell^{\geo})\stackrel{(a)}{=}\dim_\ell\Gamma_\ell^{\geo}\stackrel{(b)}{=}\dim_\ell\bar\Gamma_\ell^{\geo}\stackrel{(\ref{ldim})}{=}\dim \mathbf{G}_\ell^{\geo}=:g.
\end{equation}

\textit{\textbf{Step V}}.
Let $\Delta_\ell$ be a maximal compact subgroup of 
$\mathbf{G}_\ell^{\sc}(\Q_\ell)$ that contains $\pi^{-1}(\Gamma_\ell^{\geo})$.
By \cite[3.2]{Ti79}, $\Delta_\ell$ is the stabilizer $\mathbf{G}_\ell^{\sc} (\Q_\ell)^x$ of a vertex $x$  in 
the \emph{Bruhat-Tits building} of $\mathbf{G}_\ell^{\sc} /\Q_\ell$.   
There exists a smooth affine group scheme $\cG$ over $\Z_\ell$
and an isomorphism $\iota$ from the generic fiber of $\cG$ to $\mathbf{G}_\ell^{\sc}$ 
such that $\iota(\cG(\Z_\ell)) = \mathbf{G}_\ell^{\sc} (\Q_\ell)^x$ \cite[3.4.1]{Ti79}. 
As $\mathbf{G}_\ell^{\sc}$ is simply-connected semisimple, 
the special fiber $\cG_{\F_\ell}$ is connected \cite[3.5.2]{Ti79}. 
The maximal compact subgroup $\Delta_\ell$ is \emph{hyperspecial} if and only if $\cG_{\F_{\ell}}$ is reductive \cite[3.8.1]{Ti79}, 
in which case it has the same root datum as the generic fiber \cite[XXII, 2.8]{SGA3}. 
Since (c) the kernel of the reduction map $r:\cG(\Z_\ell)\to \cG(\F_\ell)$ is pro-solvable and (d) $\cG$ is smooth over $\Z_\ell$,
we obtain by the properties of $\dim_\ell$ that for $\ell\gg0$,
\begin{equation}\label{moddim}
\dim_\ell r(\pi^{-1}(\Gamma_\ell^{\geo}))\stackrel{(c)}{=}\dim_\ell \pi^{-1}(\Gamma_\ell^{\geo})\stackrel{(\ref{hopedim})}{=}g=
\dim \mathbf{G}_\ell^{\sc}\stackrel{(d)}{=}\dim \cG_{\F_\ell}.
\end{equation}
Since the special fiber $\cG_{\F_\ell}$ is connected, 
 $\cG_{\F_\ell}$ is semisimple for $\ell\gg0$ by (\ref{moddim}) and \cite[Thm. 4(iv)]{HL14}.
It follows from above that $\Delta_\ell$ is hyperspecial and $\cG_{\F_\ell}$ is simply-connected semisimple for $\ell\gg0$.
For any connected algebraic group $\bar{\mathbf{G}}$ of dimension $g$ defined over $\F_\ell$, 
the order of $\bar{\mathbf{G}}(\F_\ell)$ satisfies 
\begin{equation}\label{Norder}
(\ell-1)^g\leq |\bar{\mathbf{G}}(\F_\ell)|\leq (\ell+1)^g
\end{equation}
by \cite[Lem 3.5]{No87}. Hence, 
there exists a constant $c(g)\geq 1$ depending only on $g$ such that for $\ell\gg0$,
\begin{equation}\label{compare}
\frac{(\ell-1)^g}{c(g)}\leq |r(\pi^{-1}(\Gamma_\ell^{\geo}))| \stackrel{subgp}{\leq} 
|\cG(\F_\ell)|\stackrel{(\ref{Norder})}{\leq} (\ell+1)^g,
\end{equation}
where the first inequality follows by considering (\ref{moddim}), (\ref{Norder}), the properties of $\dim_\ell$ in Step IV, 
and the orders of finite simple groups of Lie type in characteristic $\ell$ \cite[$\mathsection9$]{St67}
(e.g., $|\mathrm{PSL}_k(\ell)|=\frac{1}{(k,\ell-1)}\ell^{k(k-1)/2}(\ell^2-1)(\ell^3-1)\cdots(\ell^k-1)$). 
Since $g:=\dim\mathbf{G}_\ell^{\geo}\leq n^2$ for all $\ell$, the index
$[\cG(\F_\ell):r(\pi^{-1}(\Gamma_\ell^{\geo}))]\leq C(n)$ (a constant depending only on $n$) for $\ell\gg0$.
Since $\cG_{\F_\ell}$ is simply-connected semisimple, 
$\cG(\F_\ell)$ is generated by the subset of order $\ell$ elements $\cG(\F_\ell)[\ell]$ when $\ell\gg0$ (see the proof of \cite[Thm. 4]{HL14}). 
Since $\cG(\F_\ell)[\ell]$ belongs to $r(\pi^{-1}(\Gamma_\ell^{\geo}))$ for $\ell\gg C(n)$, 
the equality $r(\pi^{-1}(\Gamma_\ell^{\geo}))=\cG(\F_\ell)$ holds for $\ell\gg C(n)$.
Therefore, the subgroup $\pi^{-1}(\Gamma_\ell^{\geo})\subset \cG(\Z_\ell)$ surjects onto $\cG(\F_\ell)$ under 
the reduction map $r$ for $\ell\gg0$.
By the main theorem of \cite{Va03}, this implies $\pi^{-1}(\Gamma_\ell^{\geo})=\cG(\Z_\ell)=\Delta_\ell$ for $\ell\gg 0$,
which is hyperspecial maximal compact in $\mathbf{G}_\ell^{\geo}(\Q_\ell)$.
\end{proof}

\begin{cor}\label{unramified}
For all sufficiently large $\ell$, the identity component of the algebraic geometric monodromy group $\mathbf{G}_\ell^{\geo}$
is unramified over $\Q_\ell$.
\end{cor}

\begin{proof}
Since a connected reductive group $\mathbf{G}/\Q_\ell$ is unramified if and only if $\mathbf{G}(\Q_\ell)$ contains a hyperspecial 
maximal compact subgroup \cite[$\mathsection1$]{Mi92}, $\mathbf{G}_\ell^{\sc}$  is unramified for $\ell\gg0$ by Theorem \ref{main}.
Since $\pi:\mathbf{G}_\ell^{\sc}\twoheadrightarrow (\mathbf{G}_\ell^{\geo})^\circ$ is surjective, the identity component $(\mathbf{G}_\ell^{\geo})^\circ$
is unramified for $\ell\gg0$.
\end{proof}

\begin{remark}
Assuming $\Phi_\ell$ is semisimple for all $\ell$, then
$\mathbf{G}_\ell^\circ\times\C\subset\GL_{V_\ell\times\C}$ 
is independent of $\ell$ by \cite{Ch04} and \cite{Ka99} (see Step V of Theorem \ref{inv}).
Corollary \ref{unramified} is a necessary condition 
for the existence of a common $\Q$-form of $\{\mathbf{G}_\ell^\circ\subset\GL_{V_\ell}\}_\ell$.
\end{remark}

\section{Semisimplicity}
In this section, we give two examples of $\{\Phi_\ell\}$ such that the hypothesis of Theorem \ref{main} holds.
It suffices to show by the lemma below that the restriction of $\phi_\ell$ to a normal subgroup of $\pi_1^{et}(X_{\bar\F_q})$ 
(i.e., by taking a connected Galois \'etale cover of $X$) is 
semisimple for $\ell\gg0$.

\begin{lemma}\cite[Lemma 3.6]{HL15b}
Let $F$ be a field, $G$ a finite group, $H$ a normal subgroup of $G$ such that $[G:H]$ is non-zero in $F$,
and $V$ a finite dimensional $F$-representation of $G$. Then $V$ is semisimple if and only if its restriction
to $H$ is so.
\end{lemma}

\noindent\textbf{Example 1.} Suppose the fibers of $f:Y\to X$ are curves or abelian varieties. 
Then the hypothesis of Theorem \ref{main} holds.

\begin{proof} \textit{\textbf{Step I.}} When $X$ is a curve, $\phi_\ell$ is factored through by 
a Galois representation of $K(X)$, the function field of $X$.
When the fibers of $f$ are abelian varieties, the conclusion follows directly from the Tate conjecture of abelian varieties over 
function fields \cite{Za74a,Za74b} (see also \cite[Ch. VI$\mathsection3$]{FW84},\cite[Thm. 3.1(iii)]{LP95}).
When the fibers are curves, the conclusion follows from above and
the fact that a smooth curve 
and its Jacobian variety have isomorphic $H^1$ representations.

\textit{\textbf{Step II.}} For general $X$, we may first assume $\mathbf{G}_\ell^{\geo}$ is connected for all $\ell$ by taking a
connected Galois \'etale cover \cite[Prop. 2.2(i)]{LP95}. By \cite{Ka99} and \cite{Ch04} (see Step V of Theorem \ref{inv}),
there exists a smooth geometrically connected curve $C$ of $X$ such that the algebraic geometric monodromy group associated to 
$Y\times_X C\to C$ is also equal to $\mathbf{G}_\ell^{\geo}$ for all $\ell$. 
Denote the images of $\pi_1^{et}(C_{\bar\F_q})$ and $\pi_1^{et}(X_{\bar\F_q})$ in respectively $\GL(V_\ell)$ and $\GL(\bar V_\ell)$ by
\begin{align*}
\begin{split}
\Lambda_\ell^{\geo}\subset\Gamma_\ell^{\geo}\subset\GL(V_\ell);\\
\bar\Lambda_\ell^{\geo}\subset\bar\Gamma_\ell^{\geo}\subset\GL(\bar V_\ell).
\end{split}
\end{align*}
We may assume $\bar\Gamma_\ell^{\geo}$ is generated by 
its order $\ell$ elements for $\ell\gg0$ by \cite{CT15}.
Since $\pi^{-1}(\Lambda_\ell^{\geo})\subset\pi^{-1}(\Gamma_\ell^{\geo})$ are compact subgroups of $\mathbf{G}_\ell^{\sc}(\Q_\ell)$ and 
$\pi^{-1}(\Lambda_\ell^{\geo})$ is hyperspecial maximal compact in $\mathbf{G}_\ell^{\sc}(\Q_\ell)$ for $\ell\gg0$ by Step I and Theorem \ref{main},
we have $\pi^{-1}(\Lambda_\ell^{\geo})=\pi^{-1}(\Gamma_\ell^{\geo})$ for $\ell\gg0$.
Hence, the index $[\Gamma_\ell^{\geo}:\Lambda_\ell^{\geo}]$ is bounded by some constant 
$C$ (depending on $n=\dim V_\ell$) for $\ell\gg0$.
It follows that $[\bar\Gamma_\ell^{\geo}:\bar\Lambda_\ell^{\geo}]\leq C$ for $\ell\gg0$ 
via the reduction map $\GL(H^i(Y_{\bar x_0},\Z_\ell))\to\GL(\bar V_\ell)$. 
This implies that the order $\ell$ elements of $\bar\Gamma_\ell^{\geo}$ belong to $\bar\Lambda_\ell^{\geo}$ when $\ell\gg C$.
Since $\bar\Gamma_\ell^{\geo}$ is generated by 
its order $\ell$ elements and $\bar\Lambda_\ell$ is semisimple on $\bar V_\ell$ for $\ell\gg0$, $\bar\Gamma_\ell^{\geo}$ 
is semisimple on $\bar V_\ell$ for $\ell\gg0$.
\end{proof}

\noindent\textbf{Example 2.} Identify $\Gamma_\ell^{\geo}$ as a subgroup of $\GL(H^i(Y_{\bar x_0},\Z_\ell))=\GL_n(\Z_\ell)$ for $\ell\gg0$.
Suppose there exists a connected semisimple subgroup $\mathbf{G}\subset\GL_{n,\Q}$ such that  
$(\mathbf{G}_\ell^{\geo})^\circ=\mathbf{G}\times\Q_\ell$ in $\GL_{n,\Q_\ell}$ and
$$\Gamma_\ell^{\geo}\cap(\mathbf{G}_\ell^{\geo})^\circ\subset \mathbf{G}(\Z_\ell)\subset \GL_n(\Z_\ell)$$
for $\ell\gg0$. Then the hypothesis of Theorem \ref{main} holds.

\begin{proof}
\textit{\textbf{Step I.}} 
By taking a connected Galois \'etale cover, we may assume $\mathbf{G}_\ell^{\geo}$ is connected for all $\ell$ \cite[Prop. 2.2(i)]{LP95}
and $\bar\Gamma_\ell^{\geo}=(\bar\Gamma_\ell^{\geo})^+$ for $\ell\gg0$ \cite{CT15}.
The closed subgroup $\mathbf{G}\subset\GL_{n,\Q}$ can be extended to a closed subgroup scheme 
$\mathbf{G}_{\Z[\frac{1}{N}]}\subset \GL_{n,\Z[\frac{1}{N}]}$ smooth over $\Z[\frac{1}{N}]$ 
for some sufficiently divisible integer $N$.
Let $\mathbf{G}_{\F_\ell}\subset\GL_{n,\F_\ell}$ be the base change to $\F_\ell$ for $\ell\gg0$.
Since $\bar\Gamma_\ell^{\geo}\subset\mathbf{G}_{\F_\ell}$ for $\ell\gg0$, we obtain 
\begin{equation}\label{compare2}
\dim_{\bar\Q_\ell}(V_\ell^{\otimes m}\otimes\bar\Q_\ell)^{\mathbf{G}}\stackrel{Lem.~\ref{summand}}{\leq}
\dim_{\bar\F_\ell}(\bar V_\ell^{\otimes m}\otimes\bar\F_\ell)^{\mathbf{G}_{\F_\ell}}\leq
\dim_{\bar\F_\ell}(\bar V_\ell^{\otimes m}\otimes\bar\F_\ell)^{\bar\Gamma_\ell^{\geo}}
\end{equation}
for $\ell\gg0$.
Since $\dim_{\Q_\ell}(V_\ell^{\otimes m})^{\pi_1^{et}(X_{\bar\F_q})}=\dim_{\bar\Q_\ell}(V_\ell^{\otimes m}\otimes\bar\Q_\ell)^{\mathbf{G}}$ as
$\Gamma_\ell^{\geo}$ is Zariski dense in $\mathbf{G}$, we obtain 
\begin{equation}\label{eq1}
\dim_{\bar\F_\ell}(\bar V_\ell^{\otimes m}\otimes\bar\F_\ell)^{\mathbf{G}_{\F_\ell}}=
\dim_{\bar\F_\ell}(\bar V_\ell^{\otimes m}\otimes\bar\F_\ell)^{\bar\Gamma_\ell^{\geo}}
\end{equation}
for $\ell\gg0$ by (\ref{compare2}) and Theorem \ref{inv}.
Since $\mathbf{G}_{\F_\ell}$ is connected semisimple for $\ell\gg0$,
the natural representation $i_\ell:\mathbf{G}_{\F_\ell}\to\GL(\bar V_\ell\otimes\bar\F_\ell)$ 
is semisimple for $\ell\gg0$ \cite{La95b}. 
Hence, it suffices to prove that for all $\ell\gg0$, 
the restriction of any irreducible $\bar\F_\ell$-subrepresentation $W_{\bar\F_\ell}$ of $i_\ell$ to $\bar\Gamma_\ell^{\geo}$ 
is still irreducible as $\bar\Gamma_\ell^{\geo}\subset \mathbf{G}_{\F_\ell}$.

\textit{\textbf{Step II.}}
Suppose $W_{\bar\F_\ell}$ is a direct summand of $i_\ell$. Then for any $m\in\N$, we have
\begin{equation}\label{eq2}
\dim_{\bar\F_\ell}(W_{\bar\F_\ell}^{\otimes m})^{\mathbf{G}_{\F_\ell}}=
\dim_{\bar\F_\ell}(W_{\bar\F_\ell}^{\otimes m})^{\bar\Gamma_\ell^{\geo}}
\end{equation}
when $\ell$ is sufficiently large by (\ref{eq1}).
Suppose $\bar\Gamma_\ell^{\geo}$ is not irreducible on $W_{\bar\F_\ell}$.
Then there exists a $k$-dimensional subrepresentation $U_{\bar\F_\ell}$ of $\bar\Gamma_\ell^{\geo}$ and $k<\dim W_{\bar\F_\ell}\leq n$ holds.
By (\ref{eq2}) and (the proof of) Corollary \ref{sym},
\begin{equation}\label{eq3}
\dim_{\bar\F_\ell}(\mathrm{Alt}^k W_{\bar\F_\ell})^{\mathbf{G}_{\F_\ell}}=
\dim_{\bar\F_\ell}(\mathrm{Alt}^k W_{\bar\F_\ell})^{\bar\Gamma_\ell^{\geo}}
\end{equation}
holds when $\ell\gg0$. Since $\bar\Gamma_\ell^{\geo}\subset \mathbf{G}_{\F_\ell}$ for $\ell\gg0$,
\begin{equation}\label{eq4}
(\mathrm{Alt}^k W_{\bar\F_\ell})^{\mathbf{G}_{\F_\ell}}=
(\mathrm{Alt}^k W_{\bar\F_\ell})^{\bar\Gamma_\ell^{\geo}}
\end{equation}
holds when $\ell\gg0$. Since $\bar\Gamma_\ell^{\geo}$ is generated by its order $\ell$ elements for $\ell\gg0$,
$\mathrm{Alt}^k U_{\bar\F_\ell}$ is one-dimensional and belongs to $(\mathrm{Alt}^k W_{\bar\F_\ell})^{\bar\Gamma_\ell^{\geo}}$
when $\ell\gg0$ by construction. Thus, $\mathrm{Alt}^k U_{\bar\F_\ell}\subset (\mathrm{Alt}^k W_{\bar\F_\ell})^{\mathbf{G}_{\F_\ell}}$
by (\ref{eq4}) which is impossible. Indeed, let $\{v_1,...,v_k\}$ be a basis of $U_{\bar\F_\ell}$ and $Z_{\bar\F_\ell}\neq 0$ a 
complement of $U_{\bar\F_\ell}$ in $W_{\bar\F_\ell}$. Since $\mathbf{G}_{\F_\ell}$ is irreducible on $W_{\bar\F_\ell}$,
there exists $x\in \mathbf{G}_{\F_\ell}(\bar\F_\ell)$ that does not preserve $U_{\bar\F_\ell}$. Then we have the following equations
\begin{align*}
\begin{split}
x\cdot v_1 &= u_1+z_1\\
x\cdot v_2 &= u_2+z_2\\
&\vdots \\
x\cdot v_k &= u_k+z_k,
\end{split}
\end{align*}
where the notation is defined so that $u_i\in U_{\bar\F_\ell}$ and $z_i\in Z_{\bar\F_\ell}$ for $1\leq i\leq k$.
We may assume $\{z_1,...,z_h\}$ is a non-empty maximal linearly independent subset of $\{z_1,...,z_k\}$.
If $\mathrm{Alt}^k U_{\bar\F_\ell}\subset (\mathrm{Alt}^k W_{\bar\F_\ell})^{\mathbf{G}_{\F_\ell}}$, then 
\begin{equation}\label{eq5}
x\cdot (v_1\wedge\cdots\wedge v_k)=(u_1+z_1)\wedge \cdots\wedge (u_k+z_k)\in \mathrm{Alt}^k U_{\bar\F_\ell}.
\end{equation}
Since $z_1\neq 0$, we obtain $k>1$ by (\ref{eq5}). Since we have the decomposition 
\begin{equation}\label{decomp}
\mathrm{Alt}^k(U_{\bar\F_\ell}\oplus Z_{\bar\F_\ell})=\bigoplus_{i+j=k}\mathrm{Alt}^iU_{\bar\F_\ell}\otimes \mathrm{Alt}^jZ_{\bar\F_\ell},
\end{equation}
we have $z_1\wedge\cdots\wedge z_k=0$ by (\ref{eq5}), which is the same as $h<k$. We may 
assume $z_{h+1},...,z_k$ are all equal to zero by 
the fact that $\{z_1,...,z_h\}$ is a maximal linearly independent subset of $\{z_1,...,z_k\}$
and replacing $\{v_1,...,v_h,v_{h+1},...,v_k\}$ 
with a suitable basis $\{v_1,...,v_h,v_{h+1}',...,v_k'\}$. It follows that $\{z_1,...,z_h,u_{h+1},...,u_k\}$ 
is linearly independent because $x$ is invertible. 
Therefore, $ u_{h+1}\wedge\cdots\wedge u_k\wedge z_1\wedge\cdots\wedge z_h$ is non-zero, which is absurd by
$z_{h+1}=\cdots=z_k=0$, (\ref{eq5}), and (\ref{decomp}).
This implies that $\bar\Gamma_\ell^{\geo}$ cannot have a $k$-dimensional subrepresentation of $W_{\bar\F_\ell}$ when $\ell$ is sufficiently large. Since there are finitely many $k$ less than $\dim W_{\bar\F_\ell}\leq n$, 
we conclude that $\bar\Gamma_\ell^{\geo}$ is irreducible on $W_{\bar\F_\ell}$ and thus semisimple on $\bar V_\ell\otimes\bar\F_\ell$
if $\ell$ is sufficiently large.
\end{proof}


\begin{thebibliography}{99}
\bibitem[Ac06]{Ac06}
  Achter, Jeffrey D.:
	The distribution of class groups of function fields,
	\textit{J. Pure Appl. Algebra}, \textbf{204}(2): 316-333, 2006.


\bibitem[AP07]{AP07}
	Achter, Jeffrey D.; Pries, Rachel:
	The integral monodromy of hyperelliptic and trielliptic curves, 
	\textit{Math.\ Ann.} \textbf{338} (2007), no. 1, 187--206.
	
\bibitem[BGK03]{BGK03} 
  Banaszak, G.; Gajda, W.; Kraso\'n, P.: 
	On Galois representations for abelian varieties with
complex and real multiplications, \textit{J. Number Theory} \textbf{100} (2003), no. 1, 117--132.

\bibitem[BGK06]{BGK06} 
  Banaszak, G.; Gajda, W.; Kraso\'n, P.: 
	On the image of $l$-adic Galois representations for abelian
varieties of type I and II, \textit{Doc. Math.} 2006, Extra Vol., 35--75.

\bibitem[BGK10]{BGK10} 
  Banaszak, G.; Gajda, W.; Kraso\'n, P.: 
	On the image of Galois $l$-adic representations for abelian
varieties of type III, \textit{Tohoku Math. J.} (2) \textbf{62} (2010), no. 2, 163--189.
	
\bibitem[CT15]{CT15}
  Cadoret, Anna; Tamagawa, Akio:
	On the geometric image of $\F_\ell$-linear representations of etale fundamental groups, preprint.
	
\bibitem[Ch97]{Ch97}	
	Chavdarov,  Nick: 
	The generic irreducibility of the numerator of the zeta function in a
  family of curves with large monodromy, \textit{Duke Math. J.}, \textbf{87}(1): 151-180, 1997.
	
\bibitem[Ch04]{Ch04}
  Chin, Chee Whye:
	Independence of $\ell$ of monodromy groups,
	\textit{J.A.M.S.}, Volume \textbf{17}, Number 3, 723-747.
	
\bibitem[Co03]{Co03}
  Conrad, Brian:
	\textit{Cohomological Descent}, online notes, \url{http://math.stanford.edu/~conrad/papers/hypercover.pdf}. 

\bibitem[dJ96]{dJ96} 
  de Jong, Aise Johan: 
	Smoothness, semi-stability and alterations,  
	\textit{Publ. Math. I.H.E.S.} no. \textbf{83} (1996), 51-93.

\bibitem[De71]{De71} 
  Deligne, Pierre: 
  Th\'eorie de Hodge II, 
  \textit{Publ. Math. I.H.E.S.}, \textbf{40} (1971), 5-57.

\bibitem[De74a]{De74a} 
  Deligne, Pierre: 
  Th\'eorie de Hodge III, 
  \textit{Publ. Math. I.H.E.S.}, \textbf{44} (1974), 5-77.
	
\bibitem[De74b]{De74b} 
  Deligne, Pierre: 
  La conjecture de Weil I, 
  \textit{Publ. Math. I.H.E.S.}, \textbf{43} (1974), 273-307.
	
\bibitem[De80]{De80} 
  Deligne, Pierre: 
  La conjecture de Weil II, 
  \textit{Publ. Math. I.H.E.S.}, \textbf{52} (1980), 138-252.

\bibitem[EHK12]{EHK12}
  Ellenberg, Jordan; Hall, Chris; Emmanuel: 
	Expander graphs, gonality, and variation of Galois representations,
   \textit{Duke Math. J.} \textbf{161} (2012), no. 7, 1233-1275.




\bibitem[Fa83]{Fa83}
  Faltings, Gerd:
  Endlichkeitss$\ddot{\mathrm{a}}$tze f$\ddot{\mathrm{u}}$r abelsche Variet$\ddot{\mathrm{a}}$ten $\ddot{\mathrm{u}}$ber Zahlk$\ddot{\mathrm{o}}$rpern. 
  \textit{Invent. Math.} \textbf{73} 1983, no. 3, 349--366.
	
\bibitem[FK87]{FK87} 
  Freitag, Eberhard; Kiehl, Reinhardt: 
  \textit{Etale Cohomology and the Weil Conjecture}, 
	Springer-Verlag (New York), 1987.


\bibitem[FW84]{FW84} 
  Faltings, Gerd; W$\ddot{\mathrm{u}}$stholz, Gisbert (eds.):
  \textit{Rational Points}, Seminar Bonn/Wuppertal 1983-1984, Vieweg 1984.
	
\bibitem[Fu11]{Fu11} 
  Fu, Lei: 
  \textit{Etale Cohomology Theory}, 
	Nankai Tracts in Mathematics - Vol. 13, World Scientific, 2011.
	
\bibitem[Ga83]{Ga83} 
  Gabber, Ofer: 
	Sur la torsion dans la cohomologie $l$-adique d\'une vari\'et\'e,  
	\textit{C. R. Acad. Sc. Paris} \textbf{297} S\'erie I (1983), 179-182. 

\bibitem[Ha08]{Ha08}
  Hall, Chris:
  Big symplectic or orthogonal monodromy modulo $\ell$, 
	\textit{Duke Math. J.}, Vol. \textbf{141}, No. 1 (2008), 179-203. 

\bibitem[Ha11]{Ha11}
  Hall, Chris:
  An open-image theorem for a general class of abelian varieties, with an appendix by E. Kowalski, 
	\textit{Bull. Lond. Math. Soc.}, Vol. \textbf{43}, No. 4 (2011), 703-711. 

\bibitem[Hu14]{Hu14}
	Hui, Chun Yin:
	$\ell$-independence for compatible systems of (mod $\ell$) Galois representations,
	\textit{Compos. Math.}, Vol. \textbf{151}, No. 7 (2015), 1215-1241.
	
	
\bibitem[HL14]{HL14} 
  Hui, Chun Yin; Larsen, Michael: 
  Type A images of Galois representations and maximality,
  arXiv:1305.1989. 
	
\bibitem[HL15a]{HL15a} 
  Hui, Chun Yin; Larsen, Michael: 
  Adelic openness without the Mumford-Tate conjecture,
 preprint.

\bibitem[HL15b]{HL15b} 
  Hui, Chun Yin; Larsen, Michael: 
  Maximality of Galois actions for abelian varieties, in preparation.

\bibitem[Ka83]{Ka83}
  Katz, Nicholas M.: 
	Wild ramification and some problems of ``Independence of $\ell$'', 
	\textit{Amer. Jour. Math.}, Vol. \textbf{105}, No. 1(Feb. 1983), 201-227.
	
\bibitem[Ka99]{Ka99}
  Katz, Nicholas M.: 
	Space Filling curves over finite fields, 
	\textit{Math. Res. Lett.}, \textbf{6} (1999), no. 5-6, 613-624, 
	corrections in: \textbf{8} (2001), no. 5-6, 689-691.


\bibitem[Ko06a]{Ko06a}
  Kowalski, Emmanuel: 
	On the rank of quadratic twists of elliptic curves over function fields,
   \textit{International J. of Number Theory} \textbf{2} (2006), p. 267-288.

\bibitem[Ko06b]{Ko06b}
  Kowalski, Emmanuel: 
  Weil numbers generated by other Weil numbers and torsion fields of
  abelian varieties, \textit{J. London Math. Soc. (2)} \textbf{74} (2006), no. 2, 273-288.

\bibitem[Ko06c]{Ko06c}
  Kowalski, Emmanuel: 
	The large sieve, monodromy and zeta functions of curves,
   \textit{J. reine angew. Math.}, \textbf{601}
(2006), 29-69.

\bibitem[Ko08]{Ko08}
  Kowalski, Emmanuel: 
	The large sieve, monodromy and zeta functions of algebraic curves II: Independence of the zeros,
   \textit{Int. Math. Res. Not.}, (2008), Article
ID rnn091, 57 pages.

\bibitem[La90s]{La90s} 
	Larsen, Michael:
	The normal distribution as a limit of generalized Sato-Tate measures, 
	preprint. 


\bibitem[La95a]{La95a} 
	Larsen, Michael:
	Maximality of Galois actions for compatible systems, 
	\textit{Duke Math.\ J.} 
	\textbf{80} (1995), no.\ 3, 601--630. 
	
\bibitem[La95b]{La95b} 
	Larsen, Michael:
	On the semisimplicity of low-dimensional representations
of semisimple groups in characteristic $p$,
	\textit{J. of Algebra} 
	\textbf{173} (1995), no.\ 2, 219-236. 	
	

\bibitem[LP90]{LP90} 
  Larsen, Michael; Pink, Richard: 
	Determining representations from invariant dimensions, \textit{Invent. Math.} \textbf{102},
377-398 (1990).



\bibitem[LP95]{LP95} 
  Larsen, Michael; Pink, Richard: 
 Abelian varieties, $l$-adic representations, and $l$-independence,
\textit{Math. Ann.} \textbf{302} (1995), no. 3, 561-579. 



\bibitem[MVW84]{MVW84}
	Matthews, C. R.; Vaserstein, L. N.; Weisfeiler, B.:
	Congruence properties of Zariski-dense subgroups I, 
	\textit{Proc.\ London Math.\ Soc. (3)} \textbf{48} (1984), no. 3, 514--532.
	
\bibitem[Mi80]{Mi80}
	Milne, James: 
	\textit{\'Etale Cohomology}, 
	Princeton University Press, 1980.
	
\bibitem[Mi92]{Mi92}
	Milne, James: 
	The points on a Shimura variety modulo a prime of good reduction, 
	The zeta functions of Picard modular surfaces, Publications du CRM, 1992, p. 151--253.

\bibitem[No87]{No87}
	 Nori, Madhav V.:
	 On subgroups of $\GL_n(\F_p)$, 
	 \textit{Invent.\ Math.}
	 \textbf{88} (1987), no.\ 2, 257--275.


\bibitem[Ri76]{Ri76}
	Ribet, Kenneth A.:
	Galois action on division points of Abelian varieties with real multiplications, 
	\textit{Amer.\ J.\ Math.} 
	\textbf{98} (1976), no.\ 3, 751--804. 	 

\bibitem[Ri85]{Ri85}
	Ribet, Kenneth A.:
	On l-adic representations attached to modular forms II, 
	\textit{Glasgow Math.\ J.} 
	\textbf{27} (1985), 185--194.	


\bibitem[Se72]{Se72}
	Serre, Jean-Pierre:
	Propri\'et\'es galoisiennes des points d\'ordre fini des courbes elliptiques, 
	\textit{Invent.\ Math.}
	\textbf{15} (1972), no.\ 4, 259--331.


\bibitem[Se85]{Se85}
	Serre, Jean-Pierre:
	Letter to J. Tate, Jan. 2, 1985.

\bibitem[Se86]{Se86} 
  Serre, Jean-Pierre:
	Lettre \'a Marie-France Vign\'eras du 10/2/1986, reproduced in \textit{Coll. Papers}, vol. IV, no. 137.
	
\bibitem[Se94]{Se94}
	Serre, Jean-Pierre:
	Propri\'et\'es conjecturales des groupes de Galois motiviques et des repr\'esentations $l$-adiques. Motives (Seattle, WA, 1991), 377--400, Proc. Sympos. Pure Math., 55, Part 1, Amer. Math. Soc., Providence, RI, 1994. 
	
\bibitem[Se98]{Se98} 
 Serre, Jean-Pierre: 
 \textit{Abelian $l$-adic representation and elliptic curves}, 
Research Notes in Mathematics Vol. 7 (2nd ed.), \textit{A K Peters} (1998).

\bibitem[SGA1]{SGA1}
 Grothendieck, A. \textit{et al}:
 \textit{S\'eminaire de G\'eometrie Alg\'ebrique} 1, \textit{Rev$\mathrm{\hat{e}}$tements \'Etales et Groupe Fondamental}, Lecture Notes in
 Math. \textbf{224}, Springer-Verlag (1971).

\bibitem[SGA3]{SGA3}
  Grothendieck, A. \textit{et al}:
	 \textit{S\'eminaire de G\'eometrie Alg\'ebrique} 3 (with M. Demazure),
	\textit{Sch\'emas en groupes}, 
	Lecture Notes in Mathematics \textbf{153}, 
	Springer-Verlag, (1970).
		
\bibitem[SGA4]{SGA4}
 Grothendieck, A. \textit{et al}:
 \textit{S\'eminaire de G\'eometrie Alg\'ebrique} 4 (with M. Artin and J. L. Verdier), \textit{Th\'eorie des Topos et Coho-
mologie \'Etale des Sch\'emas}, Lecture Notes in Math. \textbf{269}, \textbf{270}, \textbf{305},
Springer-Verlag (1972-1973).

\bibitem[SGA$4\frac{1}{2}$]{SGA4.5}
 Grothendieck, A. \textit{et al}:
 \textit{S\'eminaire de G\'eometrie Alg\'ebrique} $4\frac{1}{2}$ (by P. Deligne et al.), \textit{Cohomologie \'Etale}, Lecture Notes in Math.
\textbf{569}, Springer-Verlag (1977).

\bibitem[SGA5]{SGA5}
 Grothendieck, A. \textit{et al}:
 \textit{S\'eminaire de G\'eometrie Alg\'ebrique} 5 (with I. Bucur, C. Houzel, L. Illusie, J.-P. Jouanonou and J.-P.
Serre), \textit{Cohomologie l-adique et Fonctions L}, Lecture Notes in Math.
\textbf{589}, Springer-Verlag (1977).

\bibitem[St67]{St67}
Steinberg, Robert:
	\textit{Lectures on Chevalley groups}, Yale University, 1967.

\bibitem[Ti79]{Ti79}
Tits, Jacques:
	Reductive groups over local fields. Automorphic forms, representations and L-functions, Part 1, 	
	pp.\ 29--69, Proc. Sympos. Pure Math., XXXIII, Amer. Math. Soc., Providence, R.I., 1979.	
	
\bibitem[Va03]{Va03}
	Vasiu, Adrian: 
	Surjectivity criteria for p-adic representations I, 
	\textit{Manuscripta Math.} \textbf{112} (2003), no.\ 3, 325--355. 

\bibitem[Yu96]{Yu96}	
  Yu, Jiu-Kang:
	Toward a proof of the Cohen-Lenstra conjecture in the function field case,
  preprint (1996).	
	
\bibitem[Za74a]{Za74a}	
  Zarhin, Yu. G.:
	Isogenies of abelian varieties over fields of finite characteristics, \textit{Math.
USSR Sbornik} \textbf{24} (1974), 451-461.
	
\bibitem[Za74b]{Za74b}	
  Zarhin, Yu. G.:
	A remark on endomorphisms of abelian varieties over function fields
of finite characteristics, \textit{Math. USSR Izv.} \textbf{8} (1974), 477-480.
	
	
	
\end{thebibliography}
\end{document}